\PassOptionsToPackage{dvipsnames}{xcolor}
\documentclass[11pt, a4paper]{article}
\usepackage[left=0.5in, right=0.5in, top=0.5in, bottom=0.5in]{geometry}

\usepackage{graphicx}

\usepackage[intlimits]{amsmath}     % Zusaetzliche Matheumgebungen
\usepackage{amssymb}                % Mathematische Symbole
\usepackage{amsfonts}
\usepackage{amsthm}
\numberwithin{equation}{section}

\usepackage{mathtools}
\usepackage{hyperref}

\usepackage[capitalize,sort&compress,noabbrev]{cleveref}
\crefname{section}{section}{sections}
\crefname{subsection}{subsection}{subsections}
\Crefname{section}{Section}{Sections}
\Crefname{subsection}{Subsection}{Subsections}

\Crefname{figure}{Figure}{Figures}

\crefformat{equation}{\textup{#2(#1)#3}}
\crefrangeformat{equation}{\textup{#3(#1)#4--#5(#2)#6}}
\crefmultiformat{equation}{\textup{#2(#1)#3}}{ and \textup{#2(#1)#3}}
{, \textup{#2(#1)#3}}{, and \textup{#2(#1)#3}}
\crefrangemultiformat{equation}{\textup{#3(#1)#4--#5(#2)#6}}%
{ and \textup{#3(#1)#4--#5(#2)#6}}{, \textup{#3(#1)#4--#5(#2)#6}}{, and \textup{#3(#1)#4--#5(#2)#6}}

\Crefformat{equation}{#2Equation~\textup{(#1)}#3}
\Crefrangeformat{equation}{Equations~\textup{#3(#1)#4--#5(#2)#6}}
\Crefmultiformat{equation}{Equations~\textup{#2(#1)#3}}{ and \textup{#2(#1)#3}}
{, \textup{#2(#1)#3}}{, and \textup{#2(#1)#3}}
\Crefrangemultiformat{equation}{Equations~\textup{#3(#1)#4--#5(#2)#6}}%
{ and \textup{#3(#1)#4--#5(#2)#6}}{, \textup{#3(#1)#4--#5(#2)#6}}{, and \textup{#3(#1)#4--#5(#2)#6}}

\crefdefaultlabelformat{#2\textup{#1}#3}

\crefname{algo}{Algorithm}{Algorithms}

\usepackage{pgfplots}
\pgfplotsset{compat=1.18}
\usepgfplotslibrary{statistics}
\usepackage{float}
\usepackage{subcaption}
\usepackage{placeins} % provides \FloatBarrier

\usepackage{algorithm}%
\usepackage{algorithmicx}%
\usepackage{algpseudocode}%

\usepackage{hyphenat}
\usepackage{definitions}

\definecolor{viridisYellow}{RGB}{253,231,37}
\definecolor{viridisGreen}{RGB}{94,201,98}
\definecolor{viridisTeal}{RGB}{33,145,140}
\definecolor{viridisBlue}{RGB}{59,82,139}
\definecolor{viridisViolet}{RGB}{68,1,84}

\colorlet{PlotColorScalar}{viridisTeal}
\colorlet{PlotColorFlux}{viridisGreen}
\colorlet{PlotColorFull}{viridisYellow}

\usepackage{xargs}

\newtheorem{theorem}{Theorem}[section]
\newtheorem{lemma}[theorem]{Lemma}

\newtheorem{proposition}[theorem]{Proposition}

\newtheorem{assumption}{Assumption}[section]

\title{
  Construction of local reduced spaces for Friedrichs' systems via randomized training
  \thanks{The authors acknowledge funding by the BMBF under contract 05M20PMA
and by the Deutsche Forschungsgemeinschaft under Germany’s Excellence
Strategy EXC 2044 390685587, Mathematics M\"unster: Dynamics --
Geometry -- Structure.}}

\author{Christian Engwer, Mario Ohlberger, Lukas Renelt}

\newenvironment{pgfplotslegend}[1][]{
  \begingroup
  \csname pgfplots@init@cleared@structures\endcsname
  \pgfplotsset{#1}
}
{
  \csname pgfplots@createlegend\endcsname
  \endgroup
}

\begin{document}

\makeatletter
\pgfplotsset{
  boxplot/hide outliers/.code={
    \def\pgfplotsplothandlerboxplot@outlier{}%
  }
}
\makeatother

\maketitle

\begin{abstract}
  This contribution extends the localized training approach,
  traditionally employed for multiscale problems
  and parameterized partial differential equations (PDEs) featuring locally heterogeneous coefficients,
  to the class of linear, positive symmetric operators, known as Friedrichs' operators.
  Considering a local subdomain with corresponding oversampling domain we prove the compactness
  of the transfer operator which maps boundary data to solutions on the interior domain.
  While a Caccioppoli-inequality quantifying the energy decay to the interior holds true
  for all Friedrichs' systems, showing a compactness result for
  the graph-spaces hosting the solution is additionally necessary.
  We discuss the mixed formulation of a convection-diffusion-reaction problem where the
  necessary compactness result is obtained by the Picard-Weck-Weber theorem.
  Our numerical results, focusing on a scenario involving heterogeneous diffusion fields
  with multiple high-conductivity channels, demonstrate the effectiveness of the proposed method.
\end{abstract}

\thispagestyle{plain}
\markboth{L. Renelt, C.Engwer and M.Ohlberger}{Quasi-optimal localized approximation spaces for Friedrichs' systems}

\section{Introduction}
Many applications show spatially heterogeneous parameters with local fine-scale structures
which are too small to be resolved numerically. Nevertheless, these
structures can significantly influence the global solution behavior.
Examples of such problems, known as \emph{multiscale problems}, range from fibre-reinforced structures
to composite materials and porous media, among others.
In addition, we are also interested in \emph{parameterized problems} with locally varying influence
of the parameter. If solutions for many different parameter-values are required - for instance
in PDE-constrained optimization problems, inverse problems, Monte-Carlo simulations
or optimal control problems - conventional techniques such as
the Reduced Basis Method~\cite{BennerOhlbergerCohenWillcox} prove similarly infeasible
without employing localization techniques.
\par
As of now, there exists a wide variety of methods from the
multiscale~\cite{malqvist2014LOD,abdulle2012HMM,hou1997MsFEM,efendiev2013GMsFEM},
reduced basis~\cite{OS:15,SmePat16,BEOR:17} or domain decomposition communities~\cite{GL2017,HKKR18}.
In this contribution we focus on multiscale methods utilizing \emph{local approximation spaces}
to incorporate the fine-scale structures.
In contrast to methods that separate the solution (locally) into fine- and a coarse-scale contributions,
these local spaces are designed to locally approximate the full solution.
A global solution is then obtained by solving a globally coupled problem on the coarse scale
using the local spaces.
Depending on the chosen
domain decomposition various coupling conditions are possible,
we refer to~\cite{buhr2020localizedMOR} for an overview of established methods.
\par
One way of constructing the approximation spaces is by means of a \emph{localized training} procedure.
Here, the problem is solved on an oversampling domain and subsequently restricted to
the interior target domain. One then investigates the behavior of the transfer operator
mapping arbitrary boundary values to the solution in the interior domain.
Provided that this operator is compact one can then optimally approximate its whole range
by a few selected vectors, namely its leading left singular vectors.
In practice, quasi-optimal range approximations can be obtained by repeatedly applying
the operator to 'normal-distributed' boundary conditions~\cite{buhr2018randomized}.
\par
While this method has been proven to be applicable to scalar elliptic~\cite{buhr2018randomized}
and scalar parabolic~\cite{schleussSmetana} problems, we extend the idea to the large class
of positive symmetric operators known as \emph{Friedrichs' systems} which include scalar
and non-scalar elliptic problems, as well as certain operators of hyperbolic or mixed type.
For a recent contribution on data-driven model order reduction for Friedrichs' systems
we refer the reader to~\cite{romor2023friedrichs}.
We show that for Friedrichs' systems a Caccioppoli-inequality holds which quantifies the energy decay
of solutions from the oversampling domain to the interior.
Provided that the naturally occurring solution spaces admit a compact embedding into $\Ltwo$,
we show compactness of the transfer operator. In the last section we apply the developed theory
to the elliptic case in its first-order system reformulation.
Numerical experiments for a high-conductivity channel problem demonstrate the performance of the method.

\section{Linear, positive symmetric PDE-operators}
We start by formally defining the class of linear \emph{Friedrichs' operators}~\cite{friedrichs1958} as differential operators $A: C^\infty(\Omega) \to\LtwoM$ of the form\vspace{-1em}
\begin{equation}\label{eq:friedrichsOperator}
  Au \;=\; A_0 u + \sum_{i=1}^{d} A_i \frac{\partial u}{\partial x_i}
\end{equation}
with matrix-valued functions $A_i \in [L^\infty(\Omega)]^{m\times m},\, \sum_{i=1}^d \frac{\partial A_i}{\partial x_i}\;\eqqcolon \nabla\cdot A\in [L^\infty(\Omega)]^{m\times m}$. Additionally, we require the following two properties:
\vspace{0.25em}
\begin{enumerate}\itemsep0.25em
\item $A_i \; = A_i^T$ for all $i=1,\dots,d$,
\item $A_0 + A_0^T - \nabla\cdot A \;>\; 2 \varepsilon I_d$ for some $\varepsilon  > 0$.\\
\end{enumerate}
We define the graph space $\graphSpace$ as the space of all $\LtwoM$-functions which possess a weak $A$-derivative, \ie
\begin{equation}\label{eq:graphSpace}
  \graphSpace \;\coloneqq\; \{ u\in\LtwoM \;|\; Au\in\LtwoM\}.
\end{equation}
A norm on $\graphSpace$ is defined by the graph-norm
\begin{equation}\label{eq:graphNorm}
  \norm{u}_{H(A)}^2 \;\coloneqq\; \norm{u}_{\LtwoM}^2 + \norm{Au}_{\LtwoM}^2.
\end{equation}
One quickly verifies that $H^1(\Omega)^m \subseteq \graphSpace \subseteq \LtwoM$.
To incorporate boundary conditions we define, following~\cite{ernGuermond2007}, the boundary operator $D: H(A) \to H(A)'$ by
\begin{equation}\label{eq:boundaryOperator}
  (Du)(v) \;\coloneqq\; (Au,v)_{\LtwoM} - (u,A^*v)_{\LtwoM} \qquad \text{for all}\; u,v\in H(A).
\end{equation}
For sufficiently smooth $A_i$ one has the representation
\begin{equation*}
  (Du)(v) \;=\; \int_{\partial\Omega} v^T \Dcal u \diff s,
  \qquad \Dcal\coloneqq \sum_{i=1}^d n_i A_i
\end{equation*}
where $\vec{n} = (n_1,\dots,n_d)$ denotes the unit outer normal to $\partial\Omega$. This operator $D$ is then paired with a second, non-unique \emph{admissible boundary operator} $M: H(A) \to H(A)'$ which needs to satisfy
  \begin{enumerate}
  \item $(Mu)(u) \;\geq\; 0$ for all $u\in H(A)$,
  \item $H(A) = \ker(D-M) \;+\; \ker(D+M^*)$.
  \end{enumerate}
Given such an operator $M$ we can define the closed subspace
\begin{equation*}
  H_0(A) \coloneqq \ker(D-M) \subset H(A)
\end{equation*}
and obtain the following well-posedness result:
\begin{theorem}[Well-posedness~\cite{ernGuermond2007}]
For any $f\in\LtwoM$, the problem
  \begin{equation}\label{eq:weakProblem}
    \text{Find}\;u\in H_0(A): \quad (Au,v)_{\LtwoM} = (f,v)_{\LtwoM} \qquad \forall v\in\LtwoM.
  \end{equation}
  is well-posed. Its solution $u$ is the unique minimizer of the residual energy
  \begin{equation}\label{eq:weakMinimizationProblem}
    \min_{u\in H_0(A)} \norm{Au - f}_{\LtwoM}.
  \end{equation}
\end{theorem}

Instead of working with the non-symmetric problem~\cref{eq:weakProblem} one can thus also solve the corresponding (symmetric) normal equation
\begin{equation}\label{eq:weakNormalEq}
  \text{Find}\;u\in H_0(A): \quad (Au,Av)_{\LtwoM} = (f,Av)_{\LtwoM} \qquad \forall v\in H_0(A).
\end{equation}
This approach is also known as first order system least squares, see \eg~\cite{cai1994FOSLS,broersen2015petrov}.

\section{Optimal local approximation spaces}
For coefficient fields $A_i$ exhibiting fine-scale or locally strongly varying structure, solving \cref{eq:weakProblem} globally is not feasible. We thus employ a overlapping or non-overlapping domain decomposition $\{\Omega_i\}_{i=1}^{N_V}$ of the domain $\Omega$ and aim at approximating the restricted global solution $R_i u\coloneqq u\restr{\Omega_i}$ using only local computations.
\par
However, exactly computing $R_i u$ locally is not possible as for a well-posed local problem information about $u$ on the local boundary $\partial\Omega_i$ is required - which is inaccessible without the actual global computation of $u$.
This illustrates the need for a suitable approximation of the local solution via an efficiently computable (localized), small linear subspace.
Following~\cite{babuskaLipton2011}, one therefore defines oversampling domains  $\Omega_i^* \supsetneq \Omega_i$ satisfying
\begin{equation*}
  \dist(\Omega_i, \partial\Omega_i^*) \;\eqqcolon\;\delta > 0
\end{equation*}
and considers the \emph{transfer operator} $T_i$ mapping boundary data on $\partial\Omega_i^*$ to the solution of a local problem, restricted to the interior $\Omega_i$. We then aim at approximating the \emph{range} of $T_i$ \ie the space of local solutions for all possible boundary values on $\partial\Omega_i^*$. If the transfer operator can be shown to be compact, the optimal approximation space $R_{opt}^k$ of size $k$ is given by the $k$ leading left singular vectors. Here, optimality is measured in the sense of Kolmogorov, \ie $R_{opt}^k$ is a minimizer of the Kolmogorov $N$-width
\begin{equation*}
  d_N(rg(T_i)) \;\coloneqq\; \min_{dim(R)=N} \dist(R, rg(T_i)).
\end{equation*}
The attained minimal value is given by the first neglected singular value $\sigma_{k+1}$ of $T$.
\subsection{Derivation of the local problem formulation}
We start the derivation of a local formulation of \cref{eq:weakProblem} by considering the restricted Friedrichs' operator $A: \graphSpaceOS \to L^2(\Omega_i^*)^m$. Defining the shared global boundary
$\Gamma_i \coloneqq \partial\Omega_i^* \cap \partial\Omega$ and the internal boundary
$\Gamma_i^{int} \coloneqq \partial\Omega_i^* \setminus \partial\Omega$ we can also split the boundary operators $D_i$, $M_i$ as follows
\begin{align*}
  (D_i u)(v) &= (\Dcal^{int}_i u,v)_{L^2(\Gamma_i^{int})}
               + (\Dcal u,v)_{L^2(\Gamma_i)}, \\
  (M_i u)(v) &= (\Mcal^{int}_i u,v)_{L^2(\Gamma_i^{int})}
               + (\Mcal u,v)_{L^2(\Gamma_i)}.
\end{align*}
Note that for consistency reasons we assume the local operators $D_i, M_i$ to coincide with the global operators $D,M$ on the global boundary $\Gamma_i$.
For a given \emph{admissible boundary function} $\tilde{g}\in \Bcal_i \;\coloneqq\; rg(\Dcal_i^{int}-\Mcal_i^{int})$ we define the subspaces
\begin{align*}
  H_0(A;\Omega_i^*)
  &\coloneqq \{ u\in\graphSpaceOS \;|\; (\Dcal-\Mcal)(u) = 0\}, \\
  H_0(A; \Omega_i^*; \tilde{g})
  &\coloneqq \{ u\in H_0(A; \Omega_i^*) \;|\; (\Dcal_i^{int} - \Mcal_i^{int})(u) = \tilde{g} \}.
\end{align*}
The local weak problem then reads
\begin{equation}\label{eq:localWeakProblem}
  \text{Find}\;u_i\in H_0(A; \Omega_i^*; \tilde{g}): \quad (A u_i,v)_{L^2(\Omega_i^*)} = (f,v)_{L^2(\Omega_i^*)} \qquad \forall v\in L^2(\Omega_i^*)^m.
\end{equation}
Note, that the restricted global solution $R_i u$ is an element of $H_0(A;\Omega_i^*;\tilde{g})$ if we choose $\tilde{g} \coloneqq (\Dcal_i^{int} - \Mcal_i^{int})(R_i u)$. This implies that for this specific choice of $\tilde{g}$ the restricted global solution $R_i u$ is recovered as the unique solution of \cref{eq:localWeakProblem}.
Denoting by $\tilde{\Scal}_i: \Bcal_i \to \graphSpaceOS$ the \emph{solution operator} mapping a boundary function $\tilde{g}$ to the solution $u_i$ of~\cref{eq:localWeakProblem} we define the transfer operator $T_i$ via
\begin{equation}\label{eq:transferOperator}
  T_i: \Bcal_i \to H(A; \Omega_i),
  \qquad T_i(\tilde{g}) \;\coloneqq\; \tilde{\Scal}_i(\tilde{g})\restr{\Omega_i}.
\end{equation}
As $\tilde{\Scal}_i$ (and thus $T_i$) is only affine linear due to the right-hand side, we replace it with its linear part $\Scal_i \coloneqq \tilde{\Scal}_i - \tilde{\Scal}_i(0)$ which is the solution operator of the \emph{shifted} local problem
\begin{equation}\label{eq:shiftedLocalWeakProblem}
  \text{Find}\;u_i\in H_0(A; \Omega_i^*; \tilde{g}): \quad (A u_i,v) = 0 \qquad \forall v\in L^2(\Omega_i^*)^m.
\end{equation}
%and the reformulation as a normal equation
%\begin{equation}\label{eq:shiftedLocalWeakNormalEq}
%  \text{Find}\;u_i\in H_0(A; \Omega_i^*;\tilde{g}): \quad (Au_i,Av)_{L^2(\Omega_i^*)} = \underbrace{(Au_{0,\tilde{g}},Av)}_{\eqqcolon\;f_{\Gamma}(\tilde{g},v)} \qquad \forall v\in H_0(A; \Omega_i^*\; 0).
%\end{equation}
Its image, \ie the set of all possible solutions of \cref{eq:shiftedLocalWeakProblem}, shall in the following be denoted by
\begin{equation*}
  \Hcal_i \;\coloneqq\; \operatorname{Im}(\Scal_i) \quad\subseteq H_0(A;\Omega_i^*).
\end{equation*}

\subsection{Compactness of the transfer operator}
Showing compactness of $T_i$ usually involves two steps, a compactness argument for the space of possible solutions $\Hcal_i$ and an estimate on the energy decay from the oversampling domain $\Omega_i^*$ to the interior $\Omega_i$.
The compactness argument heavily relies on the regularizing properties of $A$ which is why we formulate it as an assumption here:
\begin{assumption}\label{ass:compactness}
  $\Hcal_i$ is compactly embedded in $L^2(\Omega_i^*)^m$.
\end{assumption}
In \cref{sec:cdr} this will be shown to hold for a general convection-diffusion-reaction operator.
\begin{proposition}[Caccioppoli inequality]\label{prop:caccioppoli}
  Let $A=A^0+\sum A^i\partial_{x_i}$ be a Friedrichs operator and $u\in\Hcal_i$ a solution of~\cref{eq:shiftedLocalWeakProblem}. Then, it holds that
  \begin{equation}
    \norm{u}_{H(A; \Omega_i)} \leq
    \left(\max_i\norm{A^i}_{\infty}\right)
    \frac{2}{\operatorname{dist}(\Omega_i,\partial\Omega^*_i)}
    \norm{u}_{L^2(\Omega^*_i)}.
  \end{equation}
\end{proposition}

\begin{theorem}[Compactness of $T_i$]\label{thm:compactness}
  Let \cref{ass:compactness} hold. Then, the transfer operator $T_i$ is compact.
\end{theorem}
\begin{proof}
  Let $(g_k)_{k\in\N} \subseteq \Bcal_i(\partial\Omega_i^*)$ be an
  arbitrary bounded sequence of admissible boundary values. Continuity
  of $\Scal_i$ implies that the local solutions on the oversampling
  domain $u_k \;\coloneqq \Scal_i(g_k)$ are also bounded in
  $\graphSpaceOS$. \cref{ass:compactness} now guarantees the existence
  of a subsequence $(u_{k_j})_{j\in\N}\subseteq\Hcal_i$
  strongly-convergent in $L^2(\Omega_i^*)^m$. Using the
  Caccioppoli-inequality (Proposition~\ref{prop:caccioppoli}) we can then deduce that on the interior domain $\Omega_i$ this subsequence even converges in the stronger $\graphSpaceLocal$-norm which concludes the proof. \hfill
\end{proof}

\section{Quasi-optimal approximation spaces via localized training}
Since the computation of the optimal approximation spaces spanned by the singular vectors of $T_i$ is infeasible in practice, we will following describe the \emph{localized training} procedure~\cite{buhr2018randomized} which efficiently generates \emph{quasi-optimal} approximation spaces by repeatedly applying $T_i$ to 'random boundary conditions' $\tilde{g}$.
\par
To that end, let $\Bcal^h\subset\Bcal$ be a discretization of the boundary functions with finite dimension $N_B\in\N$. Given a basis $\Phi_B$ of $\Bcal^h$, we define the isomorphism $D_{\Phi_B}: \Bcal^h \to \R^{N_B}$ mapping a function to its coefficient representation in $\Phi_B$. Conversely, we can obtain a 'random boundary function' by sampling a vector $\underline{r}\in\R^{N_B}$ with normal distributed entries $\underline{r}_i \sim \mathcal{N}(0,1)$ and considering $D_{\Phi_B}^{-1}\underline{r} \in\Bcal^h$.
\par
\cref{alg:randomizedRangefinder} is now obtained by combining the repeated application of $T_i$ to random boundary conditions with an a-posteriori error estimator which is given as the projection error of an additional set of random solutions onto the current range approximation, see~\cite{buhr2018randomized} for details. The error of the obtained range approximation $R^n$ can then a-priori be bounded as follows:
\begin{proposition}[\cite{buhr2018randomized}]
  Let $R^n$ be the result of \cref{alg:randomizedRangefinder}. Then, for $n\geq 4$ there holds
  \begin{equation*}
    \mathbb{E}\left( \norm{T_i - P_{R^n}T_i} \right)
    \;\leq\;
    \alpha \min_{\substack{k+p=n \\ k,\,p \geq 2}}
    \left[
      \left(1+\sqrt{\frac{k}{p-1}}\right) \sigma_{k+1}
      + \frac{e\sqrt{n}}{p} \left(\sum_{j>k}\sigma_j^2\right)^{\frac{1}{2}}
    \right].
  \end{equation*}
\end{proposition}
Note, that the bound scales approximately as $\sqrt{n}\,\sigma_{n+1}$ and is thus almost optimal.
\begin{algorithm}
  \begin{algorithmic}
    \State $B \leftarrow \emptyset$
    \State sample $\underline{r}_i \sim \mathcal{N}(0,1)$
    \State $M \leftarrow \{ TD^{-1}_S \underline{r}_1, \dots, TD^{-1}_S \underline{r}_{n_t}\}$
    \While{$\max_{t\in M}\norm{t} \cdot c_{\text{est}} > tol$}
    \State sample $\underline{r} \sim \mathcal{N}(0,1)$
    \State $B \leftarrow B \cup (TD^{-1}_S \underline{r})$
    \State $\text{orthonormalize}(B)$
    \State $M \leftarrow \left\{ t - P_{\operatorname{span}(B)}t \;|\; t\in M \right\}$
    \EndWhile
    \State \Return $R^n = \operatorname{span}(B)$
  \end{algorithmic}
  \caption{Adaptive randomized range approximation~\cite{buhr2018randomized}}
  \label{alg:randomizedRangefinder}
\end{algorithm}

\section{Application to diffusion problems in mixed form}\label{sec:cdr}
In this section we apply the developed theory to a classic convection-diffusion-reaction (CDR) problem
\begin{equation}\label{eq:strongCDRproblem}
  \begin{cases}
    -\nabla\cdot(D\nabla u) + \vec{b}\nabla u + cu &= f \qquad\text{in}\;\Omega, \\
    \hfill u &= g \qquad\text{on}\;\partial\Omega.
  \end{cases}
\end{equation}
with $D\in L^\infty(\Omega)^{d\times d}$ symmetric positive definite, $\vec{b}\in \Hdiv$, $c\in L^\infty(\Omega)$, $c - \tfrac{1}{2}\nabla\cdot\vec{b} \geq 0$ a.e., $g\in H^{1/2}(\partial\Omega)$.
\par
Localized training using a classic weak formulation has been investigated \ie in~\cite{buhr2018randomized}. These results can be considered as a reference for the performance of our approach based on the reformulation as a Friedrichs' system:
Introducing the diffusive flux $\sigma := -D\nabla u$ one obtains the mixed formulation
\begin{equation}\label{eq:mixedCDRproblem}
    A(\sigma,u) \coloneqq\;
    \begin{pmatrix}
      D^{-1}\sigma + \nabla u \\
      \nabla\cdot\sigma + \vec{b}\nabla u + cu
    \end{pmatrix}
    =
    \begin{pmatrix}
      0 \\ f
    \end{pmatrix}
    \quad\text{in}\;\Omega,
    \qquad
    u = g
    \quad\text{on}\;\partial\Omega.
\end{equation}
One can verify that $A$ is indeed a Friedrichs' operator with graph space $H(A)$
isomorphic to $\Hdiv\times H^1(\Omega)$~\cite{ernGuermond2006Friedrichs1}.
By choosing the boundary operator
\begin{equation*}
  \langle M(\sigma,u),(\tau,v)\rangle \coloneqq
  (\sigma\vec{n},v)_{L^2(\partial\Omega)}
  - (u,\tau\vec{n})_{L^2(\partial\Omega)}
\end{equation*}
we obtain classic Dirichlet-boundary conditions, \ie $H_0(A;\Omega)\cong\Hdiv\times H_0^1(\Omega)$.

\subsection{Weak formulation}
We consider the shifted local weak formulation \cref{eq:shiftedLocalWeakProblem} and the associated transfer operator
\begin{equation*}
  T_i: \Bcal_i \to \graphSpaceLocal.
\end{equation*}
\par
\begin{proposition}\label{prop:cdrTransferOpCompactness}
  Let $A$ be as in~\cref{eq:mixedCDRproblem} and $d\leq 3$. Then, $T_i$ is compact.\\[0.5em]
\end{proposition}
Following \cref{thm:compactness} we need to show that \cref{ass:compactness} holds for the given operator $A$. While the full graph space $H(A) \cong \Hdiv\times H^1(\Omega)$ is \emph{not} compact in $\LtwoM$ we can further characterize the solutions $u\in\Hcal$ to obtain a compact embedding of the subspace $\Hcal \subsetneq H(A)$ into $\LtwoM$:
\vspace{0.5em}
\begin{lemma}\label{lemma:l2rotation}
Let $d=3$, then for $(\sigma,u)\in\Hcal$ one has $D^{-1}\sigma\in\Hrotzero$.\\[0.5em]
\end{lemma}
\begin{proof}
  We have $D^{-1}\sigma = -\nabla u$ in $\Ltwo^d$ and thus $\nabla\times(D^{-1}\sigma) = 0$ as the curl of a gradient field is zero. The tangential trace also vanishes as can be easily seen from the identity
  \begin{equation*}
    \int_{\partial\Omega} \vec{n}\times(D^{-1}\sigma) \diff s = \int_{\Omega} \nabla\times(D^{-1}\sigma) \diff x = 0.\vspace{-1em}
  \end{equation*}
  \hfill
\end{proof}

\begin{theorem}[Picard-Weber-Weck~\cite{picard1984maxwell, weber1980maxwell, weck1974maxwell}]
  \label{thm:maxwellCompactness}
  Let $\varepsilon\in L^\infty(\Omega)^{3\times 3}$ be symmetric and uniformly positive definite and $\Omega\subseteq\R^3$ be a bounded weak Lipschitz-domain. Then, the embedding
  \begin{equation*}
    \Hrotzero\cap \varepsilon^{-1}\Hdiv\to \Ltwo^3
  \end{equation*}
  is compact.
\end{theorem}

\begin{theorem}[Rellich compactness theorem~\cite{rellich1930}]\label{thm:rellich}
  Let $\Omega\subseteq\R^d$ be a open and bounded Lipschitz-domain. Then, the embedding
  \begin{equation*}
    H^1(\Omega) \to \Ltwo
  \end{equation*}
  is compact.
\end{theorem}

\par
The proof of \cref{prop:cdrTransferOpCompactness} for $d=3$ now follows
as a direct combination of the last three statements:
For any bounded sequence $(\sigma_n,u_n)_{n\in\N} \subseteq\Hcal$
we can (due to \cref{lemma:l2rotation}) apply \cref{thm:maxwellCompactness} to the sequence
$(D^{-1}\sigma_n)_{n\in\N}$ by considering $\varepsilon \coloneqq D^{-1}$.
Denoting by $(D^{-1}\sigma_{n_j})_{j\in\N}$ the obtained converging subsequence,
boundedness of $D$ implies the convergence of $(\sigma_{n_j})_{j\in\N}$ in $\Ltwo^d$.
The claim follows by applying \cref{thm:rellich} to $(u_{n_j})_{j\in\N}$.
In dimension $d=2$ we have
\begin{equation*}
  \operatorname{rot}_{2}(\sigma) \coloneqq \partial_{x_1}\sigma_2 - \partial_{x_2}\sigma_1 \in \Ltwo
\end{equation*}
by an argument similar to \cref{lemma:l2rotation}.
Applying \cref{thm:maxwellCompactness} to the natural extension
$(\sigma_1,\sigma_2,0): \Omega \to \R^3$ then yields the desired statement.
The case $d=1$ directly follows from \cref{thm:rellich} due to the identity $\Hdiv = H^1(\Omega)$.

\section{Numerical experiments}
In this section we will numerically examine the proposed method for
the CDR-problem~\cref{eq:mixedCDRproblem}. We consider the domain
$\Omega=[0,1]^2$ and corresponding oversampling domain
$\Omega^* = [-\delta,1+\delta]^2$, $\delta > 0$ which we approximate using a structured grid $\Omega_h^*$ of quadrilaterals with edge length $h$.
\par
For the discretization of the graph-space $H(A) \cong \Hdiv\times H^1$
Lagrange-elements $\mathbb{P}^k(\Omega_h^ *)$ of order $k=1$ for the scalar variable
and Raviart-Thomas functions $RT^{k-1}(\Omega_h^*)$ for the flux variable are used.
The discrete boundary space $\Bcal$ is chosen as the trace space of $\mathbb{P}^k(\Omega_h^*)$, \ie $\Bcal \coloneqq \mathbb{P}^k(\partial\Omega_h^*)$.

\begin{table}[t]
  \begin{center}
    \renewcommand{\arraystretch}{1.4}
    \begin{tabular}{|c|c|c|c|c|c|c|c|}
    \hline
    Test case & $\Omega$ & $\delta$ & $D(x)$ & $\vec{b}$ & $c(x)$ \\
    \hline
    Pure diffusion & $[0,1]^2$ & various & $1.0$ & $(0,0)^T$ & $0$ \\
    \hline
    Full CDR & $[0,1]^2$ & $1.0$ &
    $
    \begin{cases}
      10^{2} &\; x\in\omega_{HC}, \\
      1.0 &\;\text{else.}
    \end{cases}
    $
    & $(1,1)^T$ &
    $
    \begin{cases}
      0, &\; x\in\omega_{HC} \\
      1.0 &\;\text{else.}
    \end{cases}
    $ \\
    \hline
    \end{tabular}
  \end{center}
  \caption{Parameters of the test cases}
  \label{tab:parameters}
  \vspace*{-1em}
\end{table}

\subsection{Different oversampling sizes}
First, we numerically examine possible choices of the oversampling distance $\delta$.
In most publications this distance is chosen as $\delta = \operatorname{diam}(\Omega_i)$
which amounts to an additional layer of coarse grid cells (\cref{fig:oversamplingDomain}).
In \cref{fig:oversamplingSizesPureDiffusion} the influence of the oversampling distance $\delta$
in a pure diffusion problem is depicted. Here, the size $N$ of the generated basis with
the tolerance being fixed, scales approximately like $\mathcal{O}(\delta^{-1})$,
reminiscent of the scaling in the Caccioppoli-inequality~\cref{prop:caccioppoli}.
As the approximation error $\norm{T - P_{N}T}$ converges exponentially in $N$ we infer the relation\\
\centerline{$
  \norm{T - P_{N}T} \;\in\; \mathcal{O}(\exp(-c\, \delta N)).
$}

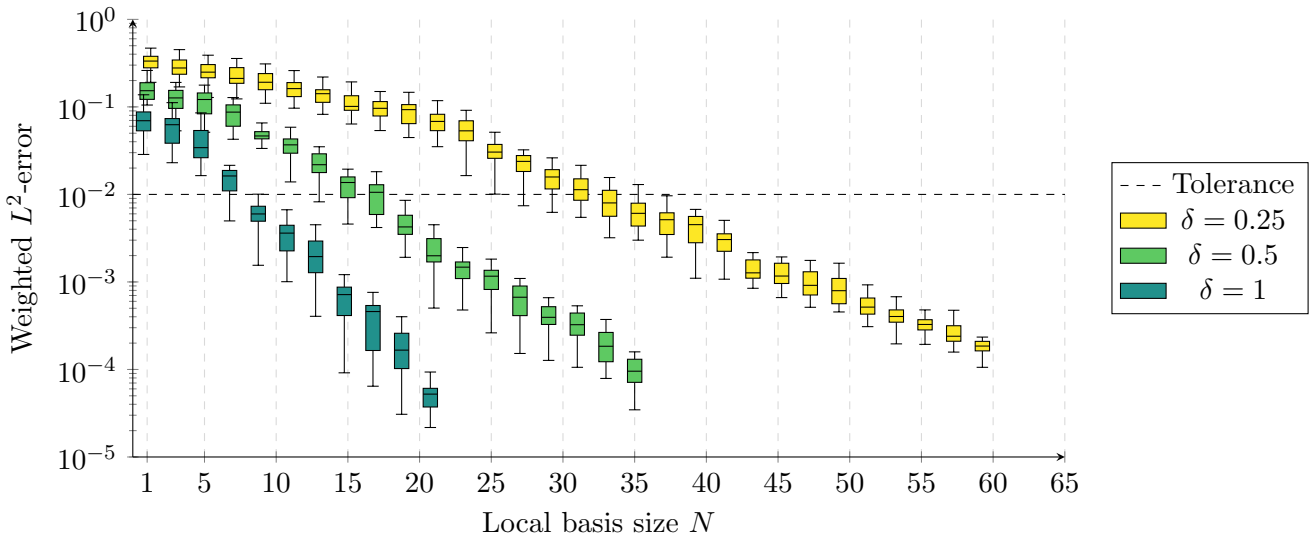
\begin{figure}[ht]
  \centering
  \def\nplots{3}
  \begin{tikzpicture}
    \begin{semilogyaxis}[
        boxplot/draw direction = y,
        boxplot/box extend = 4/(\nplots+1),
        width=0.75\linewidth,
        height=0.4\linewidth,
        xmajorgrids=true,
        grid style={gray!30, dashed},
        axis x line = bottom,
        axis y line = left,
        xmin=0, xmax=65, ymin=1e-5,ymax=1,
        xlabel= Local basis size $N$,
        xtick = {1,5,10,...,100},
        ylabel= {Weighted $L^2$-error},
        legend style={at={(1.05,0.5)},anchor=west},
      ]
      \addplot[mark=none, black, dashed, samples=2, domain=0:66] {1e-2};
      \addlegendentry{Tolerance}
      \foreach \n in {1,3,...,59} {
        \addplot[
          color=black,fill=viridisYellow,
          boxplot, /pgfplots/boxplot/hide outliers,
          boxplot/draw position = \n + (3 - 0.5*(\nplots+1))/(\nplots+1),
          forget plot,
        ]
        table[y={error_dim_\n}, fill,col sep=comma]
        {data/range_evaluation_n_60_ntest_40_neval_20_classicBasicDiffusion_weightedL2_os025.csv};
      }
      \addlegendimage{area legend, fill=viridisYellow}
      \addlegendentry{$\delta=0.25$}
      \foreach \n in {1,3,...,35} {
        \addplot[
          color=black,fill=viridisGreen,
          boxplot, /pgfplots/boxplot/hide outliers,
          boxplot/draw position = \n + (2 - 0.5*(\nplots+1))/(\nplots+1),
          forget plot,
        ]
        table[y={error_dim_\n}, fill,col sep=comma]
        {data/range_evaluation_n_80_ntest_40_neval_20_classicBasicDiffusion_weightedL2_os05.csv};
      }
      \addlegendimage{area legend, fill=viridisGreen}
      \addlegendentry{$\delta=0.5$}
      \foreach \n in {1,3,...,21} {
        \addplot[
          color=black,fill=viridisTeal,
          boxplot, /pgfplots/boxplot/hide outliers,
          boxplot/draw position = \n + (1 - 0.5*(\nplots+1))/(\nplots+1),
          forget plot,
        ]
        table[y={error_dim_\n}, fill,col sep=comma]
        {data/range_evaluation_n_120_ntest_40_neval_20_classicBasicDiffusion_weightedL2_os1.csv};
      }
      \addlegendimage{area legend, fill=viridisTeal}
      \addlegendentry{$\delta=1$}
    \end{semilogyaxis}
  \end{tikzpicture}
  \caption{\label{fig:oversamplingSizesPureDiffusion}
    Influence of different oversampling sizes $\delta$ in a simple diffusion test case. The same grid width $h=1/30$ was used in all tests.}
    \vspace*{-2em}
\end{figure}

\subsection{Results for the full CDR problem}
We now consider the full CDR problem with diffusion, advection and reaction
of comparable magnitude (\cref{tab:parameters}).
In addition, a subdomain $\omega_{HC} \subseteq \Omega^*$ of high-conductivity channels
is introduced where the diffusion tensor $D$
has a significantly larger magnitude (\cref{fig:highConductivityChannels}).
In a first test case the channels do not intersect and run parallel from left to right.
The second test case consists of channels both in $x$- and $y$-direction.
In both settings we compute approximation spaces and evaluate the projection error of
an additional set of $20$ solutions from $\Hcal$.
As the flux variable $\sigma \approx -D\nabla u$ has, due to the high diffusivity values,
a significantly larger magnitude, we consider the error norm
\begin{equation*}
  \norm{\sigma,u}^2 \;\coloneqq\; \norm{D^{-1}\sigma}_{\Ltwo^d}^2 + \norm{u}_{\Ltwo}^2
  \quad\sim\; \norm{\sigma,u}_{\LtwoM}^2.
\end{equation*}
in order to obtain a more uniform convergence of both solution components.

\def\addlegendimage{\csname pgfplots@addlegendimage\endcsname}

\definecolor{matchingRed}{HTML}{D01F3C}
\begin{figure}[ht]
  \centering
  \begin{minipage}{0.33\linewidth}
    \centering
    \begin{tikzpicture}[scale=0.9]
      \draw[step=1/12,black!20,thin] (-1,-1) grid (2, 2);
      \draw[viridisViolet,very thick] (0,0) rectangle (1,1);
      \draw[matchingRed,very thick] (-1,-1) rectangle (2,2);
      \node[viridisViolet] at (0.5,0.5) {\bf $\Omega_{i}$};
      \node[matchingRed,right] at (2.1,0.5) {\bf $\Omega_{i}^*$};
    \end{tikzpicture}
    \captionof{figure}{\label{fig:oversamplingDomain}
      Computational domain.}
    \vspace*{1.5em}
  \end{minipage}
  \hspace{0.05\linewidth}
  \begin{minipage}{0.55\linewidth}
    \centering
    \hspace{0.5mm}
    \begin{tikzpicture}[scale=0.9,baseline=(current bounding box.center)]
      \def\channelWidth{0.04}
      \fill[viridisTeal] (-1,-1) rectangle (2,2);
      \foreach \channelCenter in {-2/3,-1/3, 1/3,2/3, 4/3,5/3} {
        \filldraw[fill=viridisYellow,draw=viridisYellow]
        (-1, \channelCenter-\channelWidth) rectangle (2, \channelCenter+\channelWidth);
      }
      \draw[viridisViolet, thick] (-1,-1) rectangle (2,2);
    \end{tikzpicture}
    \hspace{1mm}
    \begin{tikzpicture}[scale=0.9,baseline=(current bounding box.center)]
      \def\channelWidth{0.04}
      \fill[viridisTeal] (-1,-1) rectangle (2,2);
      \foreach \channelCenter in {-2/3,-1/3, 1/3,2/3, 4/3,5/3} {
        \filldraw[fill=viridisYellow,draw=viridisYellow]
        (-1, \channelCenter-\channelWidth) rectangle (2, \channelCenter+\channelWidth);
        \filldraw[fill=viridisYellow,draw=viridisYellow]
        (\channelCenter-\channelWidth, -1) rectangle (\channelCenter+\channelWidth, 2);
      }
      \draw[viridisViolet, thick] (-1,-1) rectangle (2,2);
    \end{tikzpicture}
    \hspace{2mm}
    \begin{tikzpicture}[baseline=(current bounding box.center)]
      \begin{pgfplotslegend}[
        legend entries={$D_{ii}=10^2$, $D_{ii} = 1$},
        legend style={at={(0,0)}, anchor=center}
        ]
        \addlegendimage{fill=viridisYellow, area legend}
        \addlegendimage{fill=viridisTeal, area legend}
      \end{pgfplotslegend}
    \end{tikzpicture}
    \captionof{figure}{\label{fig:highConductivityChannels}
      Diagonal entries of the diffusion coefficient $D_{ii}$ exhibiting high-conductivity channels.}
  \end{minipage}
\end{figure}

In the case of parallel channels (\cref{fig:classicSolutionParallelChannels}) we first of all note the exponential decrease of
the approximation error with increasing size of the computed
local basis (\cref{fig:errorParallelChannels}). Furthermore, there is a distinctive increase
in approximation quality at $N=8$. We suspect that at this point the most energetic modes
related to the two channels reaching the interior have all been included in
the range approximation leaving only less energetic modes entering through low diffusivity regions.

% scaling of the subfigures
\def\scalingSolutionPlots{0.32}

\begin{figure}[h!]
  \centering
  \begin{subfigure}{\scalingSolutionPlots\linewidth}
    \centering
    \includegraphics[width=0.9\linewidth]{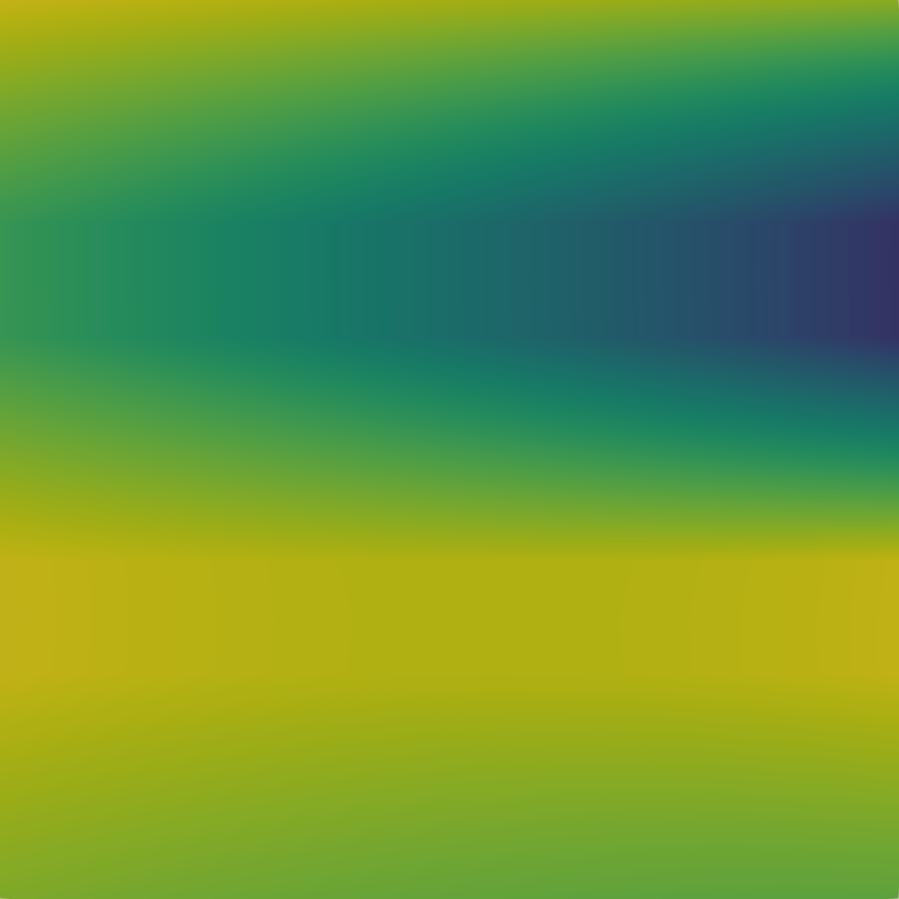}
    \caption{Restricted solution on $\Omega$}
  \end{subfigure}
  \begin{subfigure}{0.1\linewidth}
    \begin{tikzpicture}
      \pgfplotscolorbardrawstandalone[
       colormap/viridis,
       point meta min = -0.3,
       point meta max = 0.3,
       colorbar style = {height=\scalingSolutionPlots*0.85/0.1*\linewidth, width=0.25\linewidth}
      ]
    \end{tikzpicture}
    \vspace*{1.8em}
  \end{subfigure}
  \begin{subfigure}{\scalingSolutionPlots\linewidth}
    \centering
    \includegraphics[width=0.9\linewidth]{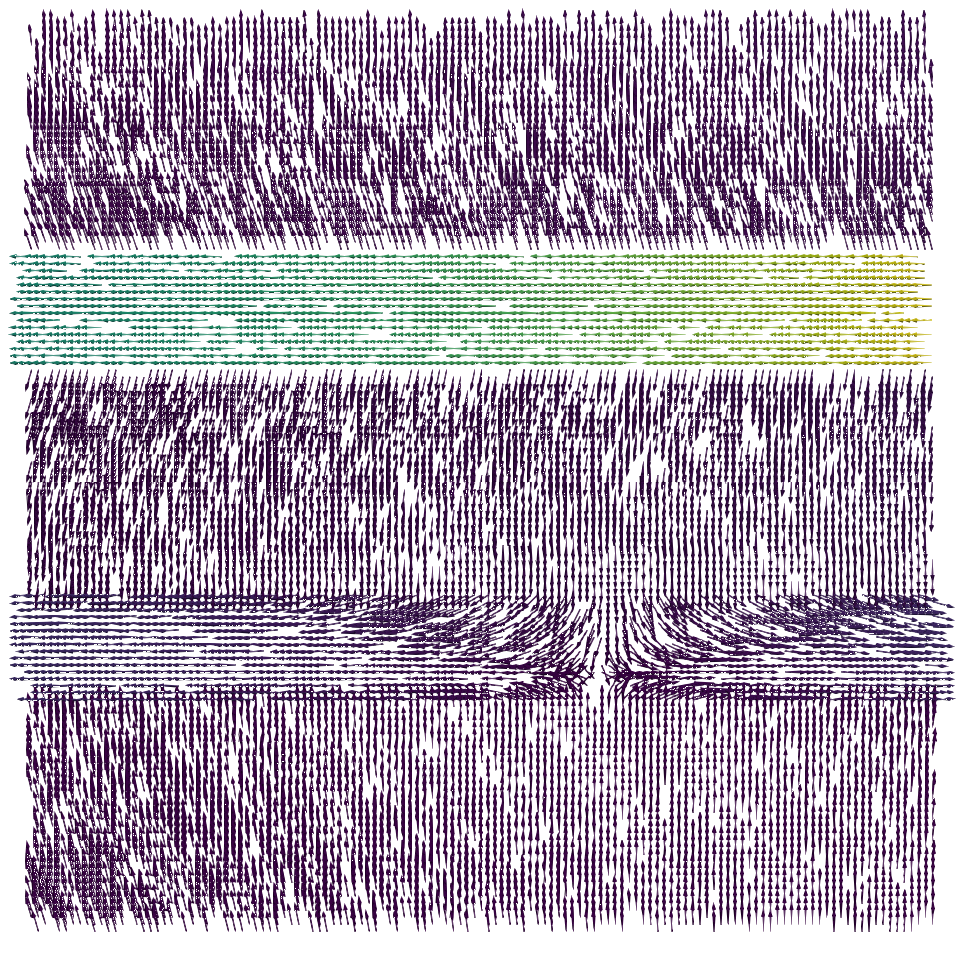}
    \caption{Restricted flux on $\Omega$}
  \end{subfigure}
  \begin{subfigure}{0.1\linewidth}
    \begin{tikzpicture}
      \pgfplotscolorbardrawstandalone[
       colormap/viridis,
       point meta min = 0.0,
       point meta max = 750.0,
       colorbar style = {height=\scalingSolutionPlots*0.9/0.1*\linewidth, width=0.25\linewidth}
      ]
    \end{tikzpicture}
    \vspace*{1.8em}
  \end{subfigure}
  \caption{\label{fig:classicSolutionParallelChannels}
    Exemplary solution to the local problem with channels in $x$-direction.}
\end{figure}

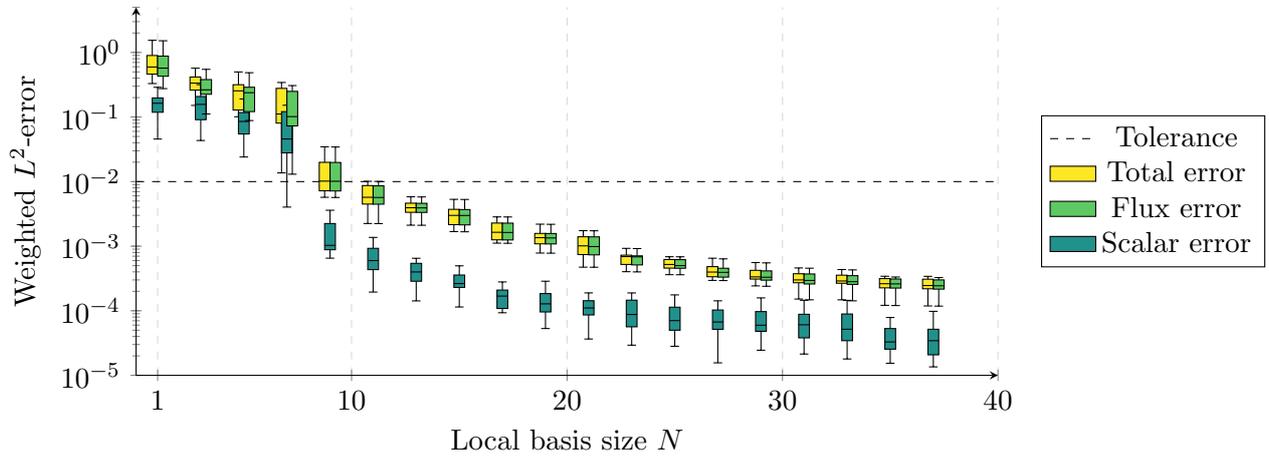
\begin{figure}[h!]
  \centering
  \def\nplots{3}
  \begin{tikzpicture}
    \begin{semilogyaxis}[
      boxplot/draw direction = y,
      boxplot/box extend = 2/(\nplots+1),
      width=0.7\linewidth,
      height=0.35\linewidth,
      xmajorgrids=true,
      grid style={gray!30, dashed},
      axis x line = bottom,
      axis y line = left,
      xmin=0, xmax=40, ymin=1e-5, ymax=5,
      xlabel= Local basis size $N$,
      xtick = {1,10,20,...,100},
      ylabel= {Weighted $L^2$-error},
      legend style={at={(1.05,0.5)},anchor=west},
      ]
      \addplot[mark=none, black, dashed, samples=2, domain=0:100] {1e-2};
      \addlegendentry{Tolerance}
      \foreach \n in {1,3,...,38} {
        \addplot[
        color=black,fill=PlotColorFull,
        boxplot, /pgfplots/boxplot/hide outliers,
        boxplot/draw position = \n + (1 - 0.5*(\nplots+1))/(\nplots+1),
        forget plot,
        ]
        table[y={error_dim_\n}, fill,col sep=comma]
        {data/range_evaluation_n_50_ntest_40_neval_20_classicParallelChannelProblem_weightedL2_contrast2.csv};
        \addplot[
        color=black,fill=PlotColorScalar,
        boxplot, /pgfplots/boxplot/hide outliers,
        boxplot/draw position = \n + (2 - 0.5*(\nplots+1))/(\nplots+1),
        forget plot,
        ]
        table[y={scalar_error_dim_\n}, fill,col sep=comma]
        {data/range_evaluation_n_50_ntest_40_neval_20_classicParallelChannelProblem_weightedL2_contrast2.csv};
        \addplot[
        color=black,fill=PlotColorFlux,
        boxplot, /pgfplots/boxplot/hide outliers,
        boxplot/draw position = \n + (3 - 0.5*(\nplots+1))/(\nplots+1),
        forget plot,
        ]
        table[y={flux_error_dim_\n}, fill,col sep=comma]
        {data/range_evaluation_n_50_ntest_40_neval_20_classicParallelChannelProblem_weightedL2_contrast2.csv};
      }
      \addlegendimage{area legend, fill=PlotColorFull}
      \addlegendentry{Total error}
      \addlegendimage{area legend, fill=PlotColorFlux}
      \addlegendentry{Flux error}
      \addlegendimage{area legend, fill=PlotColorScalar}
      \addlegendentry{Scalar error}
    \end{semilogyaxis}
  \end{tikzpicture}
  \caption{\label{fig:errorParallelChannels}
    Error in the range approximation for non intersecting parallel high-conductivity channels.}
\end{figure}

The second test case adds additional channels in the vertical direction creating a
lattice-like structure. Therefore, the different information entering from the boundary via the
channels interact which leads to a more homogeneous solution in the interior.
This also shows in a faster increasing approximation quality of the spaces produced
by \cref{alg:randomizedRangefinder} (\cref{fig:errorLatticeProblem}).

\begin{figure}[h!]
  \centering
  \begin{subfigure}{\scalingSolutionPlots\linewidth}
    \centering
    \includegraphics[width=0.9\linewidth]{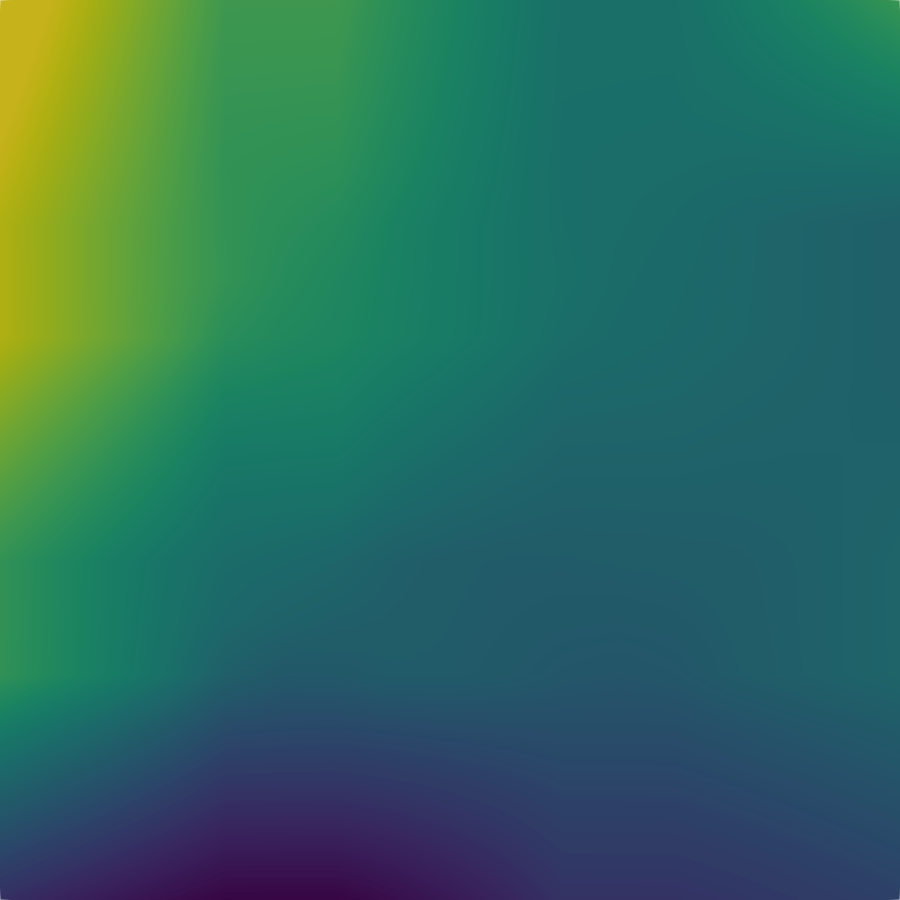}
    \caption{Restricted solution on $\Omega$}
  \end{subfigure}
  \begin{subfigure}{0.2\linewidth}
    \begin{tikzpicture}
      \pgfplotscolorbardrawstandalone[
       colormap/viridis,
       point meta min = 0,
       point meta max = 0.15,
       colorbar style = {height=\scalingSolutionPlots*0.42/0.1*\linewidth, width=0.15\linewidth}
      ]
    \end{tikzpicture}
    \vspace*{2.3em}
  \end{subfigure}
  \hfill
  \begin{subfigure}{\scalingSolutionPlots\linewidth}
    \centering
    \includegraphics[width=0.9\linewidth]{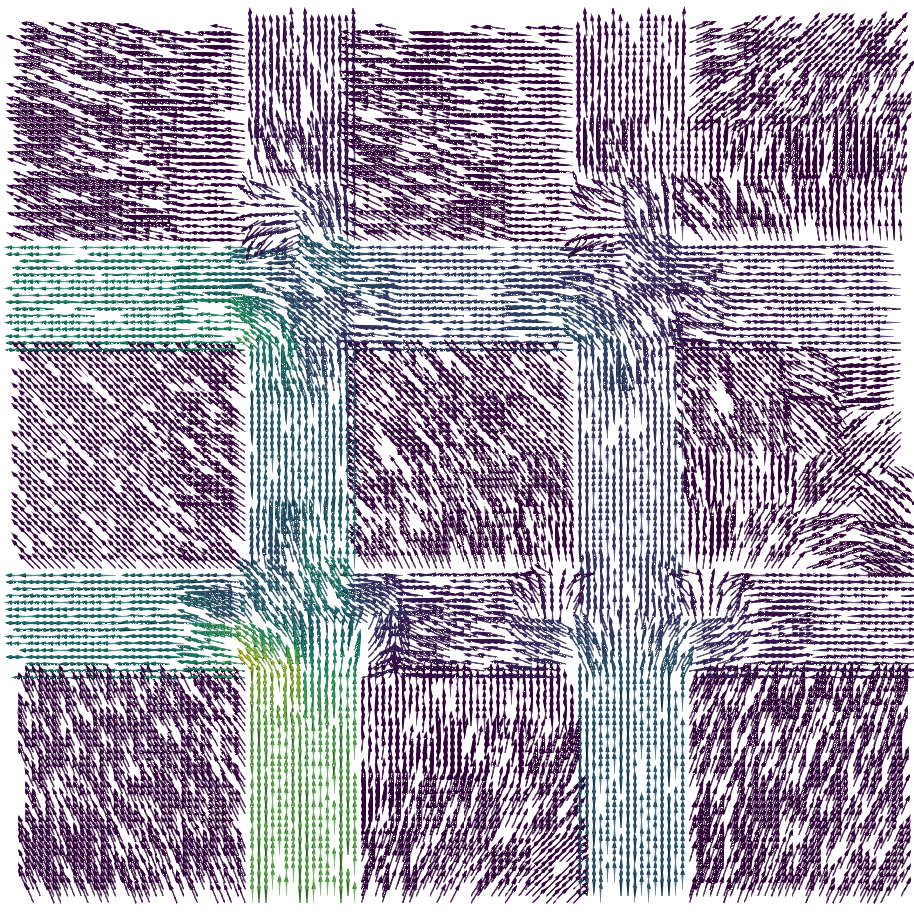}
    \caption{Restricted flux on $\Omega$}
  \end{subfigure}
  \begin{subfigure}{0.1\linewidth}
    \begin{tikzpicture}
      \pgfplotscolorbardrawstandalone[
       colormap/viridis,
       point meta min = 0.0,
       point meta max = 500.0,
       colorbar style = {height=\scalingSolutionPlots*0.9/0.1*\linewidth, width=0.3\linewidth}
      ]
    \end{tikzpicture}
    \vspace*{1.8em}
  \end{subfigure}
  \caption{\label{fig:classicSolutionLatticeChannels}
    Exemplary solution to the local problem with channels in $x$- and $y$-direction.}
\end{figure}

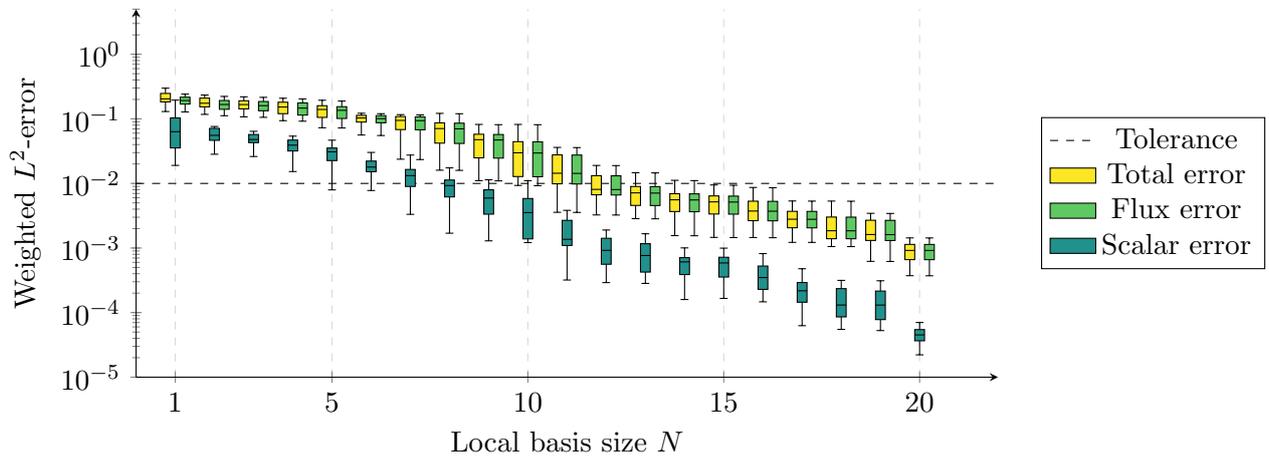
\begin{figure}[h!]
  \centering
  \def\nplots{3}
  \begin{tikzpicture}
    \begin{semilogyaxis}[
      boxplot/draw direction = y,
      boxplot/box extend = 1/(\nplots+1),
      width=0.7\linewidth,
      height=0.35\linewidth,
      xmajorgrids=true,
      grid style={gray!30, dashed},
      axis x line = bottom,
      axis y line = left,
      xmin=0, xmax=22, ymin=1e-5, ymax=5,
      xlabel= Local basis size $N$,
      xtick = {1,5,10,...,100},
      ylabel= {Weighted $L^2$-error},
      legend style={at={(1.05,0.5)},anchor=west},
      ]
      \addplot[mark=none, black, dashed, samples=2, domain=0:100] {1e-2};
      \addlegendentry{Tolerance}
      \foreach \n in {1,...,20} {
        \addplot[
          color=black,fill=PlotColorFull,
          boxplot, /pgfplots/boxplot/hide outliers,
          boxplot/draw position = \n + (1 - 0.5*(\nplots+1))/(\nplots+1),
          forget plot,
        ]
        table[y={error_dim_\n}, fill,col sep=comma]
        {data/range_evaluation_n_50_ntest_40_neval_20_classicLatticeChannelProblem_weightedL2.csv};
        \addplot[
          color=black,fill=PlotColorScalar,
          boxplot, /pgfplots/boxplot/hide outliers,
          boxplot/draw position = \n + (2 - 0.5*(\nplots+1))/(\nplots+1),
          forget plot,
        ]
        table[y={scalar_error_dim_\n}, fill,col sep=comma]
        {data/range_evaluation_n_50_ntest_40_neval_20_classicLatticeChannelProblem_weightedL2.csv};
        \addplot[
          color=black,fill=PlotColorFlux,
          boxplot, /pgfplots/boxplot/hide outliers,
          boxplot/draw position = \n + (3 - 0.5*(\nplots+1))/(\nplots+1),
          forget plot,
        ]
        table[y={flux_error_dim_\n}, fill,col sep=comma]
        {data/range_evaluation_n_50_ntest_40_neval_20_classicLatticeChannelProblem_weightedL2.csv};
      }
     \addlegendimage{area legend, fill=PlotColorFull}
      \addlegendentry{Total error}
      \addlegendimage{area legend, fill=PlotColorFlux}
      \addlegendentry{Flux error}
      \addlegendimage{area legend, fill=PlotColorScalar}
      \addlegendentry{Scalar error}
    \end{semilogyaxis}
  \end{tikzpicture}
  \caption{\label{fig:errorLatticeProblem}
    Problem involving intersecting lattice-like channels in $x$- and $y$-direction}
\end{figure}

\section{Conclusion}
In this contribution we evaluated the applicability of localized training methods to PDE-operators
of Friedrichs' type. We showed that Caccioppoli-type estimates hold and proved
compactness of the transfer operator provided that the operator exhibits sufficient
smoothing properties. The method was then applied to a convection-diffusion-reaction problem
where we verified the smoothing properties of the associated first-order Friedrichs' operator
in the weak formulation.
Employing an algorithm based on randomized numerical linear algebra then produces local spaces
that closely approximate the optimal spaces spanned by the left singular vectors
of the transfer operator. In a challenging test case featuring multiple high-conductivity channels,
we demonstrated the numerical viability of this strategy.
\par
Further research will integrate the generated local spaces into composite methods to solve
the globally coupled problem. In order to tackle parametric problems where computing
local approximation spaces for many parameters would be required, adaptive approaches
using online enrichment strategies provide a promising approach.
Notably, recent work proposed using a localized, residual-based error estimator to drive
an adaptive algorithm which was then successfully applied to
a scalar elliptic problem~\cite{keil2024local}.
Future work will aim at combining this approach with our work.

\bibliographystyle{abbrv}
\bibliography{ms}

\end{document}